\documentclass{scrartcl}

\newcounter{view} 
\usepackage{hyperref,
authblk,   
}

\usepackage{latexsym, amssymb,amsmath
}
\usepackage{cleveref} 
\usepackage{float}
\usepackage{xcolor} 

\usepackage{tikz}  
\usetikzlibrary{knots}
\usepackage{bm}

\ifnum \theview=0
\hypersetup{colorlinks=false
    }
\fi

\ifnum \theview=1
\hypersetup{colorlinks=true,
    linkcolor={red!25!black},
    citecolor={blue!40!black},
    urlcolor={blue!80!black}
    }
\fi

\ifnum \theview=2
\hypersetup{colorlinks=true,
    linkcolor={red!75!black},
    citecolor={blue!80!black},
    urlcolor={blue}
    }
\fi
\newenvironment{proof}{\vspace*{2ex}\noindent {\em Proof:}
}{\hfill $\Box$ \\[2ex]}

\newcommand{\parens}[1]{\!\left( #1 \right)}
\newcommand{\sbrace}[1]{\left[ #1 \right]}

\newcommand{\abs}[1]{\left| #1 \right|}

\makeatletter
\def\subsection{\@startsection {subsection}{2}{\z@}{.8ex
plus 1ex} {1ex}{\bfseries}}
\renewcommand\subsubsection{\@startsection{subsubsection}{3}{\z@}%
                                     {-3.25ex\@plus -1ex \@minus -.2ex}%
                                     {-1.5ex \@plus -.2ex}
                                     {\bfseries\sffamily}}
\makeatother

\numberwithin{equation}{section}
\newtheorem{remark}{Remark}[section]

\newtheorem{theorem}{Theorem}[section]
\newtheorem{proposition}{Proposition}[section]

\newtheorem{lemma}{Lemma}[section]

\newcommand{\iso}{\cong}
\newcommand{\maps}{{\,:\,}}
\newcommand{\NN}{\mathbb{N}}
\newcommand{\RR}{\mathbb{R}}
\newcommand{\QQ}{\mathbb{Q}}
\newcommand{\ZZ}{\mathbb{Z}}
\newcommand{\CC}{\mathbb{C}}
\newcommand{\tensor}{\otimes}
\newcommand{\Map}{\operatorname{Map}}

\newcommand{\tr}{\mathrm{tr}}
\newcommand{\qtr}{{\hyperref[p:qtr]{\mathrm{qtr}}}}

\newcommand{\lieg}{{\hyperref[p:lie-alg]{\mathfrak{g}}}}
\newcommand{\lieh}{{\hyperref[p:lie-alg]{\mathfrak{h}}}}
\newcommand{\Weyl}{{\hyperref[p:lie-alg]{\mathcal{W}}}}
\newcommand{\A}{{\hyperref[p:rings]{\mathcal{A}}}}
\newcommand{\s}{{\hyperref[p:small]{\mathbf{s}}}}
\newcommand{\q}{{\hyperref[p:small]{\mathbf{q}}}}
\newcommand{\DD}{{\hyperref[p:drinfeld]{\mathcal{D}}}}
\newcommand{\ad}{{\hyperref[p:invariant]{\operatorname{ad}}}}
\newcommand{\even}{{\hyperref[p:divided]{\mathrm{ev}}}}
\newcommand{\invariant}{{\hyperref[p:invariant]{\mathrm{inv}}}}
\newcommand{\restricted}{{\hyperref[p:full]{\mathrm{res}}}}
\newcommand{\full}{{\hyperref[p:full]{\mathrm{full}}}}
\newcommand{\nonrestricted}{{\hyperref[p:full]{\mathrm{nonres}}}}
\newcommand{\divv}{{\hyperref[p:divided]{\mathrm{divided}}}}
\newcommand{\divided}[1]{{\hyperref[p:divided]{\,\parens{#1}}}}
\newcommand{\weight}{{\hyperref[p:cartan]{\mathrm{weight}}}}
 \newcommand{\HC}{{\hyperref[p:HC]{\mathrm{HC}}}}
\newcommand{\cs}{{\hyperref[p:cs]{\operatorname{cs}}}}
\newcommand{\hen}{{\hyperref[p:hennings]{\operatorname{hen}}}}
\newcommand{\blam}{{\hyperref[p:Hopf0]{\sbrace{\lambda}}}}
\newcommand{\bgam}{{\hyperref[p:Hopf0]{\sbrace{\gamma}}}}
\newcommand{\Rlarge}{{\hyperref[p:trha]{R}}}

\newcommand{\Ufull}{{\hyperref[p:divided]{U}_{\full}}}
\newcommand{\Ufullm}{{\hyperref[p:divided]{U}^{\tensor m}_{\full}}}
\newcommand{\Ufullz}{{\hyperref[p:divided]{U}_{\full,0}}}
\newcommand{\Acomplete}{{\hyperref[p:Hopf0]{\widehat{\A}}}}
\newcommand{\Ucomplete}{{\hyperref[p:Hopf]{\widehat{U}}}}
\newcommand{\Ufullcomplete}{{\hyperref[p:Hopf]{\widehat{U}}_{\full}}}
\newcommand{\Ufullcompletem}{{\hyperref[p:Hopf]{\widehat{U}}^{\tensor
      m}_{\full}}}

\newcommand{\Uq}{\hyperref[p:full]{U}_q}
\newcommand{\Uqm}{\hyperref[p:full]{U}^{\tensor m}_q}
\newcommand{\Uqz}{{\hyperref[p:full]{U}_{q,0}}}
\newcommand{\Uweight}{{\hyperref[p:full]{U}_{\weight}}}
\newcommand{\Uqcomplete}{{\hyperref[p:Hopf]{\widehat{U}}_q}}
\newcommand{\Uqcompletem}{{\hyperref[p:Hopf]{\widehat{U}}^{\tensor m}_q}}
\newcommand{\Uqzcomplete}{{\hyperref[p:Hopf]{\widehat{U}}_{q,0}}}
\newcommand{\Uweightcomplete}{{\hyperref[p:Hopf]{\widehat{U}}_{\weight}}}
\newcommand{\Uweightcompletem}{{\hyperref[p:Hopf]{\widehat{U}}^{\tensor
      m}_{\weight}}}
\newcommand{\universal}{{\hyperref[p:universal]{\Gamma}}}

\newcommand{\Ik}{{\hyperref[p:Hopf]{I_{k}}}}
\newcommand{\enrich}{{\hyperref[p:small]{\Phi}}}
\newcommand{\quantuminteger}[2]{\hyperref[p:rings]{\sbrace{#1}_{#2}}}
\newcommand{\Kinteger}[1]{{\hyperref[p:cartan]{\sbrace{#1}}}}
\newcommand{\reduced}[1]{{\hyperref[p:reduced]{\parens{#1}}}}
\newcommand{\bracket}[1]{\hyperref[p:lie-alg]{\left\langle #1 \right \rangle}}
\newcommand{\lattice}{{\hyperref[p:lie-alg]{\Lambda}}}

\newcommand{\fullchamber}{{\hyperref[p:reps]{\Lambda^+}}}
\newcommand{\rcoeff}[1]{{\hyperref[p:trha]{b_{#1}}}}
\newcommand{\twist}{{\hyperref[p:universal]{r}}}
\newcommand{\clasp}{{\hyperref[p:universal]{C}}}

\begin{document}

\title{\Large Relationship of the Hennings and Chern-Simons\\
Invariants For  Higher Rank Quantum Groups}

\author{Winston Cheong}
\affil{Kansas State University (\texttt{winstonc@ksu.edu})}
\author{Alexander Doser}
\affil{Northwestern University (\texttt{alexanderdoser2023@u.northwestern.edu})}
\author{McKinley Gray} 
\affil{University of Illinois Chicago (\texttt{mgray6@uic.edu})}
\author{Stephen F. Sawin}
\affil{Fairfield University (\texttt{ssawin@fairfield.edu})}

\maketitle

\begin{abstract}
The Hennings invariant for the small quantum group associated to an
arbitrary simple Lie algebra at a root of unity is shown to agree
with Jones-Witten-Reshetikhin-Turaev 
invariant arising from Chern-Simons filed theory for the same Lie algebra and the same root of unity on all
integer homology three-spheres, at roots of unity where both are
defined.  This partially generalizes the work of Chen, et al. (\cite{CYZ12,CKP09}) 
which relates the Hennings and Chern-Simons
invariants for $\mathrm{SL}(2)$ and $\mathrm{SO}(3)$ for arbitrary
rational homology three-spheres. 
\end{abstract}

\nocite{BCL14}


\section*{Introduction}

In \cite{CKP09} Chen,  Kuppum,  and Srinivasan show that the Jones-Witten or Witten-Reshetikhin-Turaev three-manifold invariant 
 (which here is called the Chern-Simons three-manifold invariant to emphasize the light it sheds on Chern-Simons quantum field theory) associated to the Lie
group $\mathrm{SO}(3)$ at odd levels of the  $\mathrm{SO}(3)$
theory (in the sense of Dijkgraaf and Witten \cite{DW90},
corresponding to the quantum group at odd primitive  roots of unity
with root lattice representations) and the Hennings invariant
associated to the corresponding quantum group have a tight
relationship:  The Hennings invariant of a manifold is the
Chern-Simons invariant times the order of the first homology (or zero
if it is infinite order).  This result is new in the case of rational
homology three-spheres (i.e. when the order is finite), in which case
each invariant determines the other.  This result was generalized by Chen, Yu
and Zhang in 
\cite{CYZ12} to the same result for the $\mathrm{SL}(2)$ invariant at
integer levels (corresponding to the same quantum group at even roots
of unity with all representations).  Notice in each case that this
implies the invariants agree exactly on integral homology spheres.  

One would like to generalize this result to arbitrary quantum groups,
which is to say from the Chern-Simons point of view to arbitrary complex
 simple Lie groups, or at least complex simply-connected simple
Lie groups. This is not primarily because the higher-rank invariants
are interesting in themselves; indeed, both the Hennings and CS
invariants are so complicated it is difficult to imagine calculating
either invariant beyond lens spaces and perhaps $\mathrm{SL}(3)$.
However, one expects from physics that the CS invariant is
telling us deep geometric information, and it is reasonable to guess
that if the Hennings invariant is so deeply connected to the CS
invariant then it too is telling us something geometric (e.g.,
perturbative aspects of 
Chern-Simons theory).  A property 
of the CS and Hennings invariants that applies to the rank one case may be
an artifact of the algebraic simplicity of this case, while a property
of a wide range of examples is likely to be connected with the
underlying geometry of the situation, and a proof that works in a
general case might brings us closer to revealing that geometry.

There are two problems with generalizing the claim in the above two
papers to higher rank quantum groups (from here on we will be working entirely
with quantum groups, relegating mentions of Lie groups themselves
to motivation).  The first is that their argument relies heavily on a
complete knowledge of the center of the small quantum group, which is not
known for higher rank and seems like a very deep problem (though
important progress is being made on it, e.g. Lachowska (\cite{Lachowska03})). The second is that
the result as stated seems unlikely to be true! If every nilpotent
element of the center of the quantum group squares to zero, as in the
rank one case, the result still holds, but there is no reason to
believe this is the case, and if it is not the relationship between
the invariants may be quite complicated.  However, this issue does not
come up with integral homology spheres.

The strategy follows the strategy of  \cite{CKP09} and relies on the remarkable work of Habiro and L\^e \cite{HabiroLe16}. The key is to show that the universal invariant of zero linking matrix tangles at roots of unity, from which surgery presentations of homology three spheres are built, live in a particular well-behaved subspace which in the case of one-tangles is spanned by central idempotents.  This space is best described as being the image of an easily defined subspace of the generic quantum group under the extension to a root of unity (because the Harish-Chandra homomorphism is an isomorphism in the generic  case).  It would be possible to argue entirely at a root of unity that the calculation of the universal invariant ends up in this space, but in the spirit of \cite{HabiroLe16} for a very slight increase in complexity one can compute the universal invariant entirely in the generic case, where it clearly remains in the desired space, and argue that the universal invariant is preserved by the extension to roots of unity. 

Section One expands on one author's previous paper \cite{Sawin06a} to
extend and complete the restricted quantum group so that it is a
topological ribbon Hopf algebra.  It also introduces a grading on the
algebra introduced by Habiro in \cite{Habiro06} and developed by
 \cite{HabiroLe16} which is here called left degree
(We follow \cite{Habiro06} in calling the left-degree-zero component the even part,
but this is distinct from the use of the word ``even'' in \cite{CYZ12}).

Section Two reviews the universal invariant of tangles, focusing on
Habiro's (\cite{Habiro06}) notion of bottom tangles.  It also reviews
the small quantum group at a root of unity.  The fundamental new result
here is the map $\enrich$ from the topological ribbon Hopf algebra
extending and completing the restricted quantum group to the ribbon
Hopf algebra called the small quantum group, and
\autoref{pr:phi} which  says that it intertwines the universal
invariant of tangles.  Essentially, the universal invariant is the
same for the generic quantum group and the quantum group at roots of
unity.
 Section Two also reviews the representation
theory, center and invariant functionals, proving that the even part of
the center of
the completion of the restricted quantum group is spanned by elements
$z_\theta$ 
constructed from quantum traces. The two main results of this section
together give a complete description of the image of the even part of
the center under $\enrich$ as the span of these $z_\theta$.

Section Three defines the invariant functionals from which the Hennings and Chern-Simons three-manifold
invariants are constructed.  It uses results of \cite{Habiro06} (and imitates the argument of similar results in \cite{HabiroLe16})
to prove the key result that the
universal invariant of a bottom tangle with zero linking matrix is in
the even part of the tensor product of the completion of the
restricted quantum group.  Thus  when the tangle has one open component
its universal invariant  is an infinite linear combination of the
central elements $z_\theta$, and its universal invariant in the small quantum
group is a finite linear combination of their images.  This is the key insight from which the proof of the main
result follows. 

   Integral homology three-spheres are
obtained by surgery on the closure of zero linking matrix tangles with
twists added.  Thus  the computation of any three-manifold
invariant constructed from the universal invariant is reduced to a computation
involving the underlying invariant functional and these $z_\theta$.
This computation in fact agrees for any such three-manifold
invariant.  

This research was completed during Fairfield University's Summer REU
in Mathematics and Computational Sciences in summer 2015 and was made
possible by NSF grant NSF DMS-1358454.  The authors would like to thank an unnamed referee for corrections on the appropriate completion of the algebras. 


\section{Quantum Groups and their Completions}

\subsection{Large Quantum Group}
\label{ss:largeQG}

This subsection defines an extended version of the standard quantum
group over the ring of Laurent polynomials in a way that contains the
various standard specializations, restricts coherently to roots of
unity, and permits  defining link invariants.  In this it extends work of one author
(\cite{Sawin06a}), but the bulk of the material is standard and
largely comes from Lusztig (\cite{Lusztig90b,Lusztig90c,Lusztig93}) and sticks
closely to Lusztig's notation,  most of it appearing in a more expository
fashion in Chari and Pressley (\cite{CP94}).  Chen et al (\cite{CYZ12}) define $\widehat{U}_\zeta$
which is our $\Uq(\mathfrak{sl}_2)$ with $\zeta=q^{1/2}$,
$E=E_1,$ $F=F_1$ and $K^2=K_1$. Habiro and L\^e (\cite{HabiroLe16}) define
$U_\nu(\lieg)$ exactly as our $\Uq(\lieg)$ with $\nu=q$ and with the
opposite comultiplication and antipode.

\subsubsection{Lie algebra and Cartan matrix}\label{p:lie-alg}
Let $\{a_{i j}\}_{i,j=1}^M$ be a \emph{Cartan matrix} and $\{d_i\}_{i=1}^M$
be a minimal sequence of integers such that $d_i a_{i j}$ is positive
definite.  Let $\lieh$ be an $n$ dimensional Euclidean space and let
$\alpha_i \in \lieh^*$ for $1 \leq i \leq M$, called \emph{simple roots}, be
such that $\bracket{\alpha_i,\alpha_j}=d_i a_{i j}$ (thus the inner
product is normalized so \emph{short roots} have length squared $2$).  Let
$\Lambda_r$, the \emph{root lattice}, be the lattice in $\lieh^*$ generated
by the $\alpha_i$.  Let $\lattice$, the \emph{weight lattice}, be the lattice
in $\lieh^*$ whose inner product with each $\alpha_i/d_i$ is an
integer. Let $L$ be the smallest integer such that $L$ times any inner
product of weights is an integer.  Finally let the \emph{Weyl group} $\Weyl$ be
the group generated by reflections of $\lieh^*$ about the hyperplanes
$\{x\in \lieh^* \mid \bracket{x,\alpha_i}=0\}$.  $\Weyl$ preserves
$\lattice_r \subset \lattice$.  See Humphreys (\cite{Humphreys72}) for
detail on Lie algebras and their representations.

The orbit of $\{\alpha_i\}$ under the Weyl group is the set of
\emph{roots} $\Delta \in \lattice_r$.  Each element of $\Delta$ is
either a nonnegative or nonpositive combination of simple roots. Call
$\Delta^+$ the set $\beta_1, \ldots, \beta_N$ of positive roots. Let $\rho \in \lattice$ be half
the sum of the positive roots. The \emph{translated action} of the Weyl group
on $\lieh^*$ is given by $\lambda \mapsto \sigma(\lambda+\rho)-\rho$
for each $\sigma \in \Weyl$.

Let $\lieg$ be the complex, simple Lie algebra associated to $a_{i j}$
with Cartan subalgebra $\lieh$. 

\subsubsection{Quantum integers and base rings} \label{p:rings}
Let $\A_s=\ZZ[s,s^{-1}]$ be the ring of Laurent polynomials in an
indeterminant $s$ and let $\A_q=\ZZ[q,q^{-1}]$ be embedded in $\A_s$
via $q=s^L$. Let 
$q_i=q^{d_i}$, and also write $q_{\beta_j}=q_i$ for any positive root
$\beta_j$ of the same length as $\alpha_i$. Define
\begin{align*} \quantuminteger{n}{i}&=(q_i^n-q_i^{-n})/(q_i-q_i^{-1}) \in \A_q\\ 
\quantuminteger{n}{i}!  &=\quantuminteger{n}{i}\cdot \quantuminteger{n-1}{i} \cdots \quantuminteger{1}{i} \in \A_q\\
  \quantuminteger{\begin{matrix} m\\ n \end{matrix} }{i} &
  =\quantuminteger{m}{i}!/(\,\quantuminteger{n}{i}! \quantuminteger{m-n}{i}!) \in \A_q.
\end{align*} %

\subsubsection{The full Cartan subalgebra}\label{p:cartan}
Define the algebra
\begin{gather} \label{eq:full-cartan}
U_0(\lieg, \full,\A_s)=\Map(\lattice, \A_s)
\end{gather}
 (from here on drop the reference to
$\lieg$ as it and the Cartan matrix will remain fixed throughout) the
space of set-theoretic maps.  Define $K_\gamma,K_i^{\pm 1},\Kinteger{K_i;m},\Kinteger{K_i;m,n} \in
U_0(\full,\A_s)$ for $\gamma \in \lattice$, $1 \leq i \leq M$ and $m,n
\in \NN$ by
\begin{align}
K_{\gamma}(\lambda)&=s^{L\bracket{\lambda, \gamma}} \nonumber\\
K_i^{\pm 1}&= K_{\pm \alpha_i}\nonumber\\
\Kinteger{K_i;m}(\lambda)&=\quantuminteger{\bracket{\lambda,
    \alpha_i}/d_i+m}{i}\nonumber \\
\Kinteger{K_i;m,n}(\lambda)&= \quantuminteger{\begin{matrix} \bracket{\lambda,
    \alpha_i}/d_i +m\\ n \end{matrix} }{i}. \label{eq:rest}
\end{align}
Define $U_0(\restricted,\A_s)$ to be the subalgebra generated by
$K_i^{\pm 1}$ (or equivalently $K_\theta$ for $\theta \in \lattice_r$)
and $\Kinteger{K_i;m,n}_i$, define $U_0(\restricted,\A_q)$ to be the subalgebra
over $\A_q$ generated by the same elements,  define
$U_0(\restricted,\weight,\A_s)$ to be the subalgebra generated by
by $K_\gamma$
for $\gamma \in \lattice$ and $\Kinteger{K_i;m,n}_i$, and finally define $U_0(\nonrestricted,\A_q)$ to be
generated by $K_i^{\pm 1}$ and $\Kinteger{K_i;m}_i$. The Cartan
subalgebra of the nonrestricted quantum group \cite{CP94}[Sec. 9.2]
and of the restricted quantum group \cite{CP94}[Sec. 9.3]
clearly map homomorphically onto $U_0(\nonrestricted,\A_q)$ and
$U_0(\restricted,\A_q)$ respectively, and it is easy
to check that it has no kernel (any polynomial in the generators is
clearly nonzero on sufficiently large weights). 

\subsubsection{The full Hopf  algebra}\label{p:full}
Define $U(\full,\A_s)$ to be the algebra over $\A_s$ generated by 
$\{E_i,F_i\}_{i=1}^M$   and $\Map(\lattice,\A_s)$ subject to relations 
\begin{align}
  E_i f(\lambda)=f(\lambda-\alpha_i) E_i,  &\qquad F_i
  f(\lambda)=f(\lambda+\alpha_i) F_i,\label{eq:konef}
\\
 E_i F_j - F_j E_i & =\delta_{i,j}
 \Kinteger{K_i;0}_i,\label{eq:eonf}
\\
 \sum_{r=0}^{1-a_{i j}} (-1)^r 
 \quantuminteger{\begin{matrix} 1-a_{i j}\\
     r \end{matrix} }{i} 
 E_i^{1-a_{i j}-r} E_j E_i^r&=0
 \qquad \text{if } i \neq j,\nonumber\\ 
 \sum_{r=0}^{1-a_{i j}} (-1)^r
 \quantuminteger{\begin{matrix} 1-a_{i j}\\ r \end{matrix} }{i}
 F_i^{1-a_{i j}-r} F_j F_i^r&=0 \qquad \text{if } i \neq
 j.\label{eq:serre}
\end{align}

 $U_0(\full,\A_s)$ is naturally a subalgebra of $U(\full, \A_s)$. 
 Define $U_{\pm}(\full, \A_s)$ and $U_{\pm}(\full, \A_q)$ to be the
 subalgebras over the given rings generated by $\{E_i\}$ and
 $\{F_i\}$ respectively.  Define $U(\nonrestricted,\A_q)$ to be the
 subalgebra generated by
 $U_0(\nonrestricted,\A_q)$ and $\{E_i,F_i\}$.
 In each case  define the \emph{degree} on this
 algebra, a grading by $\lattice_r$ sending $E_i$, $f(\lambda)$ and $ F_i$ to
 $\alpha_i$, $0$ and $-\alpha_i$.

Two key facts about more standard specializations readily extend to this larger
algebra.  First, it is immediate from \autoref{eq:konef} and
\autoref{eq:eonf} that any element of $U(\full)$ or
$U(\nonrestricted)$ can be written as a linear combination of products
$\mathbf{F}\mathbf{K}\mathbf{E}$, where $\mathbf{F} \in U_-$, $\mathbf{K}\in U_0$ and $\mathbf{E} \in U_+$. In fact, if $\theta_1, \theta_2$ are sums of positive roots and $U_{\theta_1, \theta_2}$ is the span of all $\mathbf{F}\mathbf{K}\mathbf{E}$ where $\mathbf{F}$ is of degree $-\theta_1'$, $\mathbf{E}$ is of degree $\theta_2'$, and each $\theta_i-\theta_i'$ is a sum of positive roots, then the collection of all $U_{\theta_1, \theta_2}$ is a filtration making $U(\full)$ a filtered algebra.  
Second
following \cite{Lusztig90b}, the braid group of the Weyl group  acts as
automorphisms of $U(\full)$, and a choice of reduced factorization of the
longest element of the Weyl group gives a map of the Weyl group into
the braid group which in turn gives an ordering 
$\beta_1, \ldots, \beta_N$ of $\Delta^+$ and elements $E_{\beta_j},
F_{\beta_j} \in U_{\pm}$ homogeneous respectively of
degree $\pm \beta_j$. Further, writing $\vec{r}=(r_1, \ldots, r_N)$ with
$r_i =0,1,2,\ldots$, one can show that a basis for $U_{\pm}$
is given by $E_{\beta_1}^{r_1} \cdots E_{\beta_N}^{r_N}$ and $F_{\beta_1}^{r_1} \cdots
  F_{\beta_N}^{r_N}$ respectively. 

From these two facts it follows that $U(\nonrestricted,\A_q)$ is in fact the
nonrestricted specialization of the quantum group associated to
$\lieg$ in \cite{CP94}[Sec. 9.2].

\subsubsection{Divided powers and the restricted specialization}\label{p:divided}
Of course $U(\full,\A_s)$ can be extended to the field of fractions
over $\A_s$, and over that field one can define the so-called divided powers
\begin{gather}
  \label{eq:divpow} E_{\beta_j}^{\divided{r}}=E_{\beta_j}^r/\quantuminteger{q}{\beta_j}!
\end{gather}
and likewise for $F$.  Define $U(\full, \divv,\A_s)$ to be the
subalgebra generated over $\A_s$ by the divided powers
\autoref{eq:divpow} and $\Map(\lattice,
\A_s)$.  Once again $U_{\pm}(\full, \divv,\A_s)$ have as bases
\begin{align} E_{(\vec{r})}&= E_{\beta_1}^{\divided{r_1}} \cdots E_{\beta_N}^{\divided{r_N}}
  \nonumber\\ 
F_{(\vec{r})}&= F_{\beta_1}^{\divided{r_1}} \cdots
  F_{\beta_N}^{\divided{r_N}}
  \label{eq:pmbasis}
\end{align} 
and because \cite{CP94}[Sec. 9.3] argues 
\[E_i^{\divided{r}} F_i^{\divided{s}} = \sum_{0 \leq t \leq r,s} F_i^{\divided{s-t}} \Kinteger{K_i;
  2t-s-r,t}_i E_i^{\divided{r-t}}\]
it again follows that $U(\full, \divv,\A_s)$ is spanned by 
\begin{gather} \label{eq:basis}
  \mathbf{F} \mathbf{K} \mathbf{E}
\end{gather}
where $\mathbf{K} \in \Map(\lattice,\A_s)$, $\mathbf{F}=F_{\divided{\vec{r}}}$, and $\mathbf{E}=E_{\divided{\vec{s}}}$ for some $\vec{r}$ and $\vec{s}$.  Call this algebra $U_{\full}$.  It again has a filtration $U_{\theta_1, \theta_2}$ by pairs of positive elements of the root lattice.  

Similarly $U(\restricted, \divv, \A_q)$ is the subalgebra over
$\A_q$ generated by the divided powers \autoref{eq:divpow} and
$U_0(\restricted, \A_q)$, and 
$U(\restricted, \divv, \weight, \A_s)$ is the subalgebra generated 
by the divided powers and $U_0(\restricted, \weight, \A_s)$.  In each
case the algebra is spanned by \autoref{eq:basis} with $\mathbf{K}$ an element
of the associated Cartan subalgebra.  In particular as above one can
check that the restricted quantum group $U_q^{\restricted}$
\cite{CP94}[9.3] is
isomorphic to $U(\restricted, \divv, \A_q)$.  Henceforth call this
algebra $U_q$ for simplicity. Also define $U_{\weight}=U(\restricted,
\divv, \weight,\A_s)$, the restricted specialization of what \cite{CP94}
calls the simply connected version of the quantum group.

 For the algebras labeled $\restricted$ and $\nonrestricted,$ define a
 second grading, the  \emph{left degree}, valued in $\lattice_r/2\lattice_r$
 (or $\lattice/2\lattice$ if labeled by $\weight$).  It assigns the
 weights $r\alpha_i$, $0$,$-\lambda$, $0$, and $0$ to $E_i^{\divided{r}}$, $F_i^{\divided{r}}$, $K_\lambda$, $\Kinteger{K_i,m}$,
 and $\Kinteger{K_i;m,n}$ respectively.  The degree $0$ subalgebra of each of these
 algebras is called \emph{even} or $\even$, so for instance $\Uq^{\even}=U(\restricted,
 \divv, \even,\A_q)$.
\begin{remark}
\cite{Habiro06} as well as \cite{HabiroLe16} define the notion of \emph{even} with the role of  $E$ and $F$ reversed,
because they are using the opposite comultiplication.  The term \emph{even}
here should be treated as identical to theirs, and the results about
the even subalgebra here are, apart from details about the ring,
algebra and completions used, as duplicates of results in those papers, reiterated in the current language to fix notation.
\end{remark}

\subsection{Topological Ribbon Hopf Algebra}
\label{ss:ribbon-hopf}

Topological Hopf algebras have a variety of definitions such as
Bonneau \& Sternheimer and Habiro \& L\^e (\cite{BonneauSternheimer05,HabiroLe16}), but
here it will be useful to give a slightly more general one. 

\subsubsection{Topological Hopf algebra}\label{p:THA-definition}
Recall following Bourbaki \cite{Bourbaki98d} a  \emph{uniform structure}   on a space $U$ is a collection of sets  $O \subset U \times U$, closed under superset, intersection and interchanging the two factors, such that each \emph{entourage} $O$ contains the diagonal and for each $O$ there is an $O'$ such that if $(x,y),(y,z) \in O'$ then $(x,z) \in O$.  This captures enough of the structure of $\epsilon$-neighborhood in a metric space to permit the definition of Cauchy sequences and uniform continuity.  A collection of $O$ which are symmetric under switching factors, contain the diagonal and have the $O'$ property such that any finite intersection contains another such is called a fundamental system of entourages and the collection of all supersets of these give a uniform structure.

 On the category of filtered spaces a  uniform structure is a uniform structure on each subspace in the filter such that the uniform structure on the smaller subspaces have the subspace uniform structure.  

If $U$ is a module over a ring $A$ we assume that addition, ring multiplication and the ring action on $U$ are all uniformly continuous. Typically such  uniform structures will have a fundamental system  described by a collection of ideals in $A$ and submodules of $U$, the entourage associated to $I$ being $\{(x,y)\,|\, x-y \in I\}$.   Let $\overline{U}$ be the completion as a module over the completion $\overline{A}$.  The tensor product $U \tensor V$ of two modules $U$ and $V$ over $A$ (tensor over $A$ will be understood) contains a canonical uniform structure generated by $O_{O_1, O_2}= \sbrace{O_1 \tensor (V \times V)} \oplus \sbrace{(U \times U) \tensor O_2}$   for $O_1,O_2$ entourages of $U$ and $V$ respectively.  The completed tensor product $\overline{U \tensor V}$ will be written as $U \overline{\tensor} V$.  

A topological (graded) Hopf algebra is a module $U$ over a ring $A$, each with a uniform structure with uniformly continuous maps $M \maps U^{\overline{\tensor} m} \to \overline{U}$, $i \maps \overline{A} \to
\overline{U}$, $\Delta \maps \overline{U}   \to U^{\overline{\tensor} m}$, $\epsilon\maps  \overline{U} \to
\overline{A}$, and $S\colon \overline{U} \to \overline{U}$, satisfying all the usual assumptions of a Hopf algebra, as enumerated in  Chari \& Pressley \cite{CP94}.  
A topological ribbon Hopf algebra is a topological Hopf algebra $U$
together with an element $R \in U \overline{\tensor} U$ called an $R$ matrix
and ribbon element $g \in \overline{U}$ satisfying the usual axioms of a ribbon
Hopf algebra as enumerated in  Reshetikhin \& Turaev, Chari \& Pressley, Bakalov \& Kirillov, and
Turaev (\cite{RT90,CP94,BakalovKirillov01,Turaev94}.

\subsubsection{The   topological Hopf algebra} \label{p:Hopf0}
 
For each $\lambda \in \lattice$  write $\blam$ for the idempotent element of $\Ufullz=\Map\parens{\lattice, \A_s}$
sending $\gamma \in \lattice$ to $\delta_{\lambda, \gamma}$.  Give $\Ufullz$ the ``$\lattice$-adic topology'', the uniform structure  generated by an ideal $I_S=\{x\,|\,\, \blam x=0, \,\, \forall \lambda \in S\}$ for each finite $S \subset \lattice$.   $\Ufullz$ is complete in this uniform structure, the completed tensor product $U_{\full,0}^{\overline{\tensor} m} = \Map\parens{\lattice^{\times m}, \A_s}$, and multiplication is easily seen to be uniformly continuous.  In fact $\Ufullz$   is a topological Hopf algebra dual to the
group Hopf algebra on $\lattice$. More explicitly 
\begin{align}
 \blam \cdot \bgam &= \delta_{\lambda, \gamma}\blam  \label{eq:h-prod}\\
  \epsilon\parens{\blam}&=\delta_{\lambda,0}\\
  S(\blam)&= -\blam\\
  \Delta(\blam)&=\sum_{\gamma} \bgam \tensor \sbrace{\lambda-\gamma} \label{eq:h-delta}
\end{align}
are all uniformly continuous.  This structure  restricts to a literal Hopf algebra structure on
$U_0(\restricted)$ and $U_0(\nonrestricted)$.

Recall from \S\ref{p:full} that $\Ufull$ is filtered by $U_{\theta_1,\theta_2}$, and that $U_{\theta_1,\theta_2} \iso \bigoplus \mathbf{F} \Ufullz \mathbf{E}$, the sum being over a finite collection  of basis elements  $\mathbf{E}$ and $\mathbf{F}$  of $U_+$ and $U_-$.  Thus each subset in the filter inherits a uniform structure as a finite direct sum of copies of $\Ufullz$, and $\Ufull$ is a completed filtered module over $\A_s$.  Each element of the completed tensor product $\Ufull^{\overline{\tensor} m}$ can be written uniquely as (for some finite $n$)
\begin{align}
  \label{eq:tensor-basis}
  \sum_{j=1}^n \mathbf{F}_j \mathbf{K}_j \mathbf{E}_j& \quad \text{where}\\
  \mathbf{K}_j &\in U_{\full,0}^{\overline{\tensor} m}= \Map\parens{\lattice^{\times m}, \A_s}\nonumber\\
  \mathbf{F}_j &= F_{(\vec{r}_1)} \tensor \cdots \tensor F_{(\vec{r}_m)}\nonumber \\
  \mathbf{E}_j &= E_{(\vec{s}_1)} \tensor \cdots \tensor E_{(\vec{s}_m)}\nonumber 
\end{align}
that is to say each $\mathbf{E}_i $ and $\mathbf{F}_i$  stand for some unique basis element in $U_+^{\tensor m}$ and $U_-^{\tensor m}$ respectively.  It is easy to check that multiplication is uniformly continuous (because each submodule in the filter is finite dimensional over $\Ufullz$, it suffices to check on a pair of entries in \autoref{eq:tensor-basis}).  Defining topological Hopf algebra structures by combining the topological Hopf algebra structure of \autoref{eq:h-prod}-\autoref{eq:h-delta} as homomorphisms and anithomomorphisms with 
\begin{align}
  \epsilon(E_i)=\epsilon(F_i)&=0\\
S(E_i)=-E_i K_i^{-1} \qquad & \qquad S(F_i)=-K_i F_i\\
\Delta(E_i)= E_i \tensor K_i + 1 \tensor E_i \quad & \quad
\Delta(F_i)= F_i \tensor 1 + K_i^{-1} \tensor F_i \label{eq:hopf}
\end{align}
(extending this to divided powers (\autoref{p:divided}) in the obvious way), $\Ufull$ becomes a graded topological Hopf algebra.  From here on in all tensor products will be completed, so for ease of notation the $\overline{\tensor}$ symbol will be replaced with $\tensor$ unless it is important to emphasize.

\subsubsection{The completed ring} \label{p:Hopf}
For each positive integer $k$ let $\{k\}=(1-q)(1-q^2)\cdots (1-q^k)$.  Following \cite{HabiroLe16} give $\A_s$ or $\A_q$ the ``$\{k\}$-adic'' uniform structure generated by ideals $\Ik=\{k\}\A_s$ or $\Ik=\{k\}\A_q$ respectively, for each $k$.  This also gives a uniform structure on each version of $U$ by modules $\{k\}U$.  The completion of $\A_s$ and $\A_q$ with respect to this uniform structure are called $\Acomplete_s$ and $\Acomplete_q$. Write $\Ufullcomplete$ and $\Ufullcompletem$ for the completions of $\Ufull$ and $\Ufullm$ via this uniform structure (here we are treating $\Ufull$ as an ordinary module, \emph{not} as a graded module), but write $\Uqcomplete$, $\Uweightcomplete$, $\Uqcompletem$, $\Uweightcompletem$ for the completions with respect only to the $\{k\}$-adic topology.  The Hopf algebra structure is uniformly continuous and extends to the completions.

Of course the $\Acomplete$ uniform structure just allows infinite sums of terms with every increasing coefficients of factorial quantum integers, so now an arbitrary element of $\Ufullcompletem$ can be written 
\begin{align}  
  \sum_{k=0}^\infty  \sum_{j=1}^{n_{k}} \{k\} \mathbf{F}_{j,k} \mathbf{K}_{j,k} \mathbf{E}_{j,k} & \quad \text{where}\label{eq:full-tensor-basis}\\
  \mathbf{K}_{j,k} &\in U_{\full,0}^{\overline{\tensor} m}= \Map\parens{\lattice^{\times m}, \A_s}\nonumber\\
  \mathbf{F}_{j,k} &= F_{(\vec{r}_1)} \tensor \cdots \tensor F_{(\vec{r}_m)}\nonumber \\
  \mathbf{E}_{j,k} &= E_{(\vec{s}_1)} \tensor \cdots \tensor E_{(\vec{s}_m)}
  \nonumber 
\end{align}
as in \autoref{eq:tensor-basis}.

By the same token an arbitrary element of $\Uqcompletem$ can be written 
\begin{gather}
  \sum_{k=0}^\infty  \sum_{j=1}^{n_{k}} \{k\}
  \label{eq:q-tensor-basis} \mathbf{F}_{j,k} \mathbf{K}_{j,k} \mathbf{E}_{j,k}
  \end{gather}
where each $\mathbf{K}_{j,k}$ is an element of $U_{q,0}^{\tensor m}$.

\subsubsection{Weakly even elements}\label{p:weak}
The definition and lemma in this subsection is needed to prove \autoref{pr:ZLMT}, which guarantees that a certain class of calculations will give a result in the even subspace.   It borrows heavily from \cite{HabiroLe16}[Prop. 7.2].

An element of $\Ufullcompletem$ is \emph{weakly even} if the product of all its tensor factors in any order is in $\Uqcomplete^{\even}$, i.e. if any permutation in $S_n$ applied to the tensor factors followed by  multiplication gives an even element in $\Uq$.  

\begin{lemma}
  \label{lm:weak}
  If $A : \Ufull \to \Ucomplete_\full^{\tensor 2}$ is the map $(1\tensor S) \Delta$, then the completion of $A^{\tensor m}\sbrace{\Ufullcompletem} \tensor \Uq^{\even,\tensor n_1}\tensor \Uq^{\mathrm{odd},\tensor n_2}$  is weakly even for every $m$, $n_1$, and for every even  $n_2$.
\end{lemma}
\begin{proof}
    It suffices to check on a topological spanning set, and in fact if it is true of an element of $\Ufullcompletem$ it is also true when some of these tensor factors are multiplied together,  so it suffices to check the result on a element of $\Ufullcompletem$ whose factors are each of the form $f(\lambda)$, $E_{\divided{\vec{r}}}$, and  $F_{\divided{\vec{r}}}$.  By \autoref{eq:hopf}, $A$ applied to the latter two is contained in $\Uq^{\even,\tensor 2} \oplus \Uq^{\mathrm{odd},\tensor 2}$.  Finally $A f(\lambda)  = \bigoplus_{\gamma} f(\lambda+\gamma) \tensor \bgam$.  Thus it suffices to check when some of the tensor factors are even homogeneous elements, an even number of them are odd homogeneous elements, and the remaining are pairs of the form $\bigoplus_{\gamma} f(\lambda+\gamma) \tensor \bgam$.  Writing these in any order, the product is unchanged if for each such pair you eliminate those two factors and multiply the result by the sum of the degrees of the interposed factors (ignoring other factors that are part of pairs, which have degree $0$.  This reduces it the case $m=0$, where it is obvious.
\end{proof}

\subsubsection{The full topological ribbon Hopf algebra} \label{p:trha}
Define $R_0 \in \Ucomplete_{\full,0}^{\tensor 2}$ by
\begin{gather}
  \label{eq:hrmatrix}
  R_0= \sum_{\lambda, \gamma \in \lattice}  q^{\bracket{\lambda,
      \gamma}} \blam \tensor \bgam
\end{gather}
(notice here  coefficients in $\A_s$ are essential).  Define $R \in \Ucomplete_{\full}^{\tensor 2}$, by
\begin{align}
  \label{eq:largermatrix}
  R&=R_0 \sum_{\vec{r}} b_{\vec{r}} E_{\divided{\vec{r}}}
  \tensor F_{\divided{\vec{r}}}\\
  b_{\vec{r}}&= \prod_{j=1}^N q_{\beta_j}^{r_j(r_j+1)/2}
  (1-q_{\beta_j}^2)^{r_j} \quantuminteger{r_j}{\beta_j}\!! \nonumber.
\end{align}
Observe that of course each $b_{\vec{r}} \in \Ik$ for all but finitely many $k$ because $(1-q_{\beta_j}^2)^{r_j} \quantuminteger{r_j}{\beta_j}\!!$ is a multiple of $\{r_j\}$. 
$\Rlarge$ is quasitriangular  and together
with the ribbon element $g=K_{2\rho}$ defines a ribbon structure.

\section{Invariants of Tangles}

\subsection{Bottom Tangles and the Universal Invariant}
\label{ss:bottomtangles}

The universal invariant was first defined by Lawrence and Hennings
 (\cite{Lawrence88,Hennings96}).  Our discussion of bottom tangles and the universal
invariant follows \cite{Habiro06}, with terminology and definitions
most closely matching \cite{HabiroLe16}.

\subsubsection{Adjoint action and invariants} \label{p:invariant}

  In a topological Hopf algebra the adjoint
action of $a$ on $\overline{U}$ is given by $\ad_a x = \sum_j a'_j x S(a''_j)$
where $\Delta a = \sum_j a_j' \tensor a_j''$, and generalizes to an
action on $\overline{U}^{\tensor m}$ by applying $\ad^{\tensor m}$ to $\Delta^m(a)$.
$\overline{U}^{\tensor m, \invariant}$ is the subspace of elements 
on which this action is trivial, i.e. $\ad_a x= \epsilon(a)
x$. 

By \autoref{lm:weak}  $\Uq^\even$ is a module for the adjoint action of
$\Ufull$.  In particular, if $\lambda \in \lattice$ and $x \in \Uq$ of
degree $\theta \in \lattice_r$ then 
\begin{gather} \label{eq:lambda-adjoint}
\ad_{\blam}(x) = \delta_{\lambda, \theta} x.
\end{gather}

Define $\overline{U}^{*,\invariant}$ to be the set of uniformly
continuous linear functionals $\overline{U} \to \overline{A}$
invariant under the coadjoint action of $\overline{U}$ 
on its dual (i.e., $\psi$ such that $\psi(\ad_a(x))= \epsilon(a)
\psi(x)$), whose elements are called \emph{invariant functionals}.
These are 
exactly the uniformly continuous functionals $\psi$ such that $\psi(a b)=\psi\parens{b
  S^2(a)}$ for $a,b\in \overline{U}$.  In this case $\phi(a)= \psi(g^{-1} a)$ is
tracial in the sense that $\phi(a b)= \phi(b a)$, and indeed any
tracial functional $\phi$ gives a invariant functional $\psi(a)= \phi(g
a)$, so tracial and  invariant functionals are in one-to-one
correspondence.

\subsubsection{Framed tangles and bottom tangles} \label{p:tangle}
A \emph{framed tangle} is a smooth embedding of $m$ oriented  unit intervals and
$n$ oriented circles (\emph{open} and \emph{closed components}) in $D=\RR\times [-1,1] \times
[0,1]$ (in pictures the $x$ and $z$ components are in the paper/screen and $y$ is perpendicular) with endpoints of the intervals lying on $\RR \times \{0\}
\times \{0\}$ and $\RR \times \{0\} \times \{1\}$ together with a
smoothly varying normal vector at each point of each component which
is pointing in the positive $x$-direction at each endpoint.  Framed
tangles are considered up to isotopy of $D$ leaving the endpoints of
the interval fixed.  A \emph{framed bottom tangle} is a framed tangle
with the initial and final endpoint of the the $j$th interval at the
points $(2j,0,0)$ and $(2j-1,0,0)$.  A \emph{projection} of a framed
tangle is a projection of an isotopy representative onto the $x z$
plane so that the framing is never normal to the projection, the map
of each component into the plane is an immersion, and
self-intersections are all transverse double points and not at
critical points of $z$. Two projections represent the same isotopy
class of a framed tangle if and only they can be connected by planar isotopy
and the framed Reidemeister moves of \autoref{fg:reidemeister}
discussed by
Burde and Zieschang in  unframed form  and Trace in framed
form (\cite{BZ86,Trace83}). 

\begin{figure}

  \begin{center}
    \includegraphics[width=5.5in]{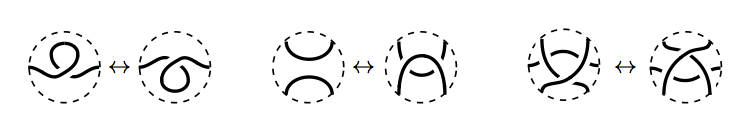}
\end{center}    
  \caption{The framed Reidemeister moves}
  \label{fg:reidemeister}
\end{figure}

If $T$ and $S$ are two tangles, write $T \tensor S$ for the tangle
formed by placing $T$ to the left of $S,$ and write $TS$ for the
tangle formed by placing $T$ above $S$ and rescaling $z$ by a factor
of $1/2$, assuming the open components of $T$ intersect the bottom
($\RR \times \{0\} \times \{0\}$) in the same locations and
corresponding orientations as the open components of $S$ intersect the
top, so that the result is again a tangle.

\subsubsection{The universal invariant} \label{p:universal}
Suppose $U$ is a topological ribbon Hopf algebra (\autoref{p:THA-definition}) with
quasitriangular $R= \sum_{k \in K} a_k \tensor b_k$ with each $a_k,
b_k \in U$ (the sum understood to converge) and ribbon element $g$.  Write $R^{-1}=
\sum_{k \in K} \tilde{a}_k \tensor \tilde{b}_k = \sum_{k \in K} S(a_k)
\tensor b_k =\sum_{k \in K} a_k \tensor S^{-1}(b_k)$.  The universal invariant of
a tangle $T$ with each closed component labeled by an invariant
functional $\universal_U(T)$ is defined as follows.  Steps (1)-(5) define the universal invariant for tangles with open components, Step (6) modifies the process for closed components
\begin{enumerate}
  \item Choose a \emph{projection} and an
    \emph{ordering of the components}, and place a
    \begin{enumerate}
    \item \emph{bead} at each local extreme point  of the height ($z$-direction) of each
      component where the
      component's orientation is pointing to the right, and a
    \item \emph{bead} on each strand of
      each crossing.
    \end{enumerate}
  \item A \emph{state} is an assignment of an element of the
    index set $K$ of the sum defining $R$ to each crossing of the projection.
    For each state, put elements of $U$ on each \emph{bead} as follows.
    \begin{enumerate}
    \item Put $g$
      on each bead at a minimum and $g^{-1}$ on each bead at a maximum.
    \item Put $a_k$ on the bead on the lower strand and $b_k$ on the bead on the
      upper strand of each positively oriented crossing, where $k$ is the
      element of $K$ associated to that crossing by the \emph{state}. 
    \item Put $\tilde{a}_k$ and $\tilde{b}_k$ on the lower and upper strands of each
      negative crossing.
    \item If both strands are pointing to the right, apply
      $S^2$ to the strand pointing up and to the right.
    \end{enumerate}
  \item Combine the beads by sliding together and replacing adjacent beads labeled $b$ and $a$ in the direction of the orientation
    with the  with a single bead labeled $ab$ \emph{until each component has a single bead}.
  \item Using the ordering of the components, the elements of $U$ on each bead  determine an element of $U^{\tensor m}$.
  \item Summing over all states gives an element $\universal_U(T)$ of the completion  $\overline{U}^{\tensor m}$ which we call the universal invariant of the tangle.
  \item If there are closed components,
    \begin{enumerate}
    \item choose a projection $P$ of an open tangle and an ordering of its components so that $P(C \tensor C \tensor \cdots \tensor C \tensor 1 \tensor \cdots \tensor 1)$ is a projection of $T$ with the first $k$ components becoming the closed components, where $C$ is the right pointing cup at the bottom left of \autoref{fg:universal} and $1$ is the tangle with one vertical component oriented up or down,
    \item $\universal(T)= (\phi_1 \tensor \cdots \phi_k \tensor 1 \tensor \cdots \tensor 1) \universal(P)$ where the $\phi_i$ are the invariant functionals labeling the closed components of $T$.
    \end{enumerate}
\end{enumerate}

  \begin{figure}[ht]
    \centering
    \includegraphics[width=5.5in]{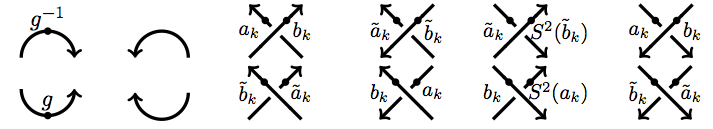}
    \caption{Labeling a tangle diagram to compute the universal
      invariant}
    \label{fg:universal}
  \end{figure}%
 Habiro   shows 
\begin{theorem}[\cite{Habiro06}] \label{th:universal} If a framed tangle
 $T$ has $m$ open components and if all closed components are labeled by
 invariant functionals  then
  the quantity $\universal_U(T)$ computed above is an element of
  $\overline{U}^{\tensor m}$ which  is independent of
the choices made in the 
construction (except for the obvious effect of reordering the open
components) and of ambient isotopy.  Further, concatenating two tangles
in the $x$ direction gives an invariant which is the tensor product of the two
invariants, and concatenating compatible tangles in the $z$ direction
(gluing open components appropriately) gives an
invariant which is the product of the invariants. Finally, if $T$ is a
bottom tangle and the components are ordered left to right then $\universal_U(T) \in \overline{U}^{\tensor m,\invariant}$.
\end{theorem}

\subsubsection{Properties of the universal invariant} \label{p:properties}

The universal invariant can be computed in very versatile ways.  Choosing an open neighborhood in plane that intersects the tangle diagram transversely with no the points which are assigned \emph{beads} in step (1) on the boundary (a \emph{fragment}), compute steps (1)-(3)   in the neighborhood and the complement separately, to get two formal sums of assignments of $U$ elements assigned to components of each.  Then sum over pairs of \emph{states} in the interior and exterior  separately, of the result of apply steps (4)-(6)  to each such pair of assignments.

Natural operations on tangles and such fragments correspond to natural operations on the
Hopf algebra.
\begin{enumerate}
  \item If one doubles one open component of a
fragment to two parallel components (the framing makes this unambiguous), the invariant of the
doubled tangle is obtained from the invariant of the original by, for each state,
applying $\Delta$ to the label on the doubled component and putting
the first factor on the rightmost of the two new components (according
to the orientation) and the second on the leftmost (interpreting the sum in $\Delta$ as breaking the state up into multiple states).
\item Reversing the
orientation of an open component has the effect of applying $S$ to
that component, and then multiplying it on the left by $g^{-1}$ if the
original component started on the bottom and on the right by $g$ if
the original component ended on the top.   
\end{enumerate}

These are both  shown in
\cite{Habiro06}.

The sequel will  frequently mention the universal invariants $r$ of the bottom
tangle representing the positive twist, $r^{-1}$ of the bottom tangle
representing the negative twist, and $C$ of the clasp tangle, all
as illustrated in \autoref{fg:some-tangles}.  The first two are
elements of the center $\overline{U}^{\invariant}$, the third of
$\overline{U}^{\tensor 2, \invariant}$.  The closure of the clasp
tangle is called the Hopf link, and if each component is labeled by an
invariant functional its universal invariant is a scalar in
$\overline{A}$. 
  \begin{figure}[ht]
    \centering
    \includegraphics[width=5.5in]{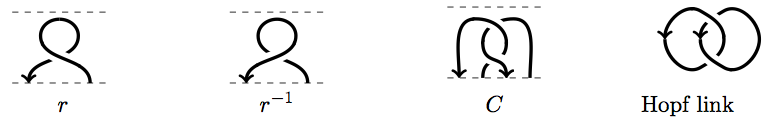}
    \caption{The positive and negative twists, the clasp tangle and
      the Hopf link}
    \label{fg:some-tangles}
  \end{figure}%

\subsection{Small Quantum Group}
\label{ss:smallQG}

The discussion of the small quantum group follows Lusztig
(\cite{Lusztig93,Lusztig90c}), with occasional small additions from
\cite{Sawin06a}.

\subsubsection{Small quantum group} \label{p:small}
Now restrict the generic $q$ to a root of unity.
Specifically, let $l$ be a positive integer and consider the
homomorphism $\A_s \to \QQ(\s),$ where $\s$ is a primitive $lL$th root
of unity (i.e. satisfies the $lL$th cyclotomic polynomial) given by
$s \mapsto \s.$  Since  $\{k\}=0$ in $\QQ(\s)$  for large enough $k$, this map
extends continuously to 
$\Acomplete_s$.  Define the extension through scalars
$U_\s=\Ufullcomplete \tensor_{\Acomplete_s} \QQ(\s)$.

For each $i \leq n$ let $l_i$ be $l/\gcd(l,d_i)$ and let $l_i'$ be
$l_i$ or $l_i/2$ according to whether $l_i$ is odd or even.  Likewise
define $l_\beta$ and $l'_\beta$ for positive roots $\beta$, and define
$l'$ to be $l$ or $l/2$ according to the parity of $l$.  Then in $\QQ(\s)$
$l'_\beta$ is the least natural number such that $[l'_\beta]_\beta=0$
and $l_\beta$ is the least such that $q_\beta^{l_\beta}=1$. Define
$\lattice_l=\{\lambda \in \lattice\,|\, \bracket{\lambda,\gamma}\in l
\ZZ,\quad \forall \gamma \in \lattice\}$. Notice in $U_\s$,  $E_\beta^{l'_\beta}=
F_\beta^{l'_\beta}=0$ and 
 $K_\lambda=1$ when $\lambda \in \lattice_l$ (in particular
 $K_i^{l_i}=1$). Therefore in $U_\s$ there are only finitely many
 distinct ideals $I_S$ in $\Ufull0$ and only finitely many distinct ideals
 $\Ik$, so the extension map completes by continuity to a map
 $\enrich\colon \Ufullcomplete \to U_\s$.

Define $u_\q$ to be the image of $\Uq \subset \Ufull$ (and hence of
$\Uqcomplete$) under $\enrich$, an algebra over
$\QQ(\s^L)=\QQ(\q)$. Likewise define $u_\s$ to be the image of
$\Uweight \subset \Ufull$.  Because extension is not exact these are
not the same as the extension of $\Uq$ and $\Uweight$ by $\QQ(\q)$
and $\QQ(\s)$ respectively.  In fact by
the previous paragraph $u_{\q}$ and $u_{\s}$  are finite dimensional.  The left degree and the even subalgebras
$u_{\s}^{\even}$ and $u_{\q}^{\even}$ are well-defined. 
It is straightforward to check that $u_\q$ and $u_\s$ are Hopf subalgebras.

\subsubsection{Small ribbon Hopf algebra} \label{p:ribbon}

The Hopf algebra $u_\s$ admits a ribbon (not just topological) Hopf
algebra structure with  
\begin{align*}
  \enrich(\Rlarge_0)&= \frac{1}{\abs{\lattice/\lattice_l}} \sum_{[\lambda], [\gamma]\in \lattice/\lattice_l} q^{-\bracket{\lambda,
      \gamma}} K_\lambda \tensor K_\gamma \\
  R&=\enrich(\Rlarge)= \enrich(\Rlarge_0) \sum_{\{\vec{r}|\, r_j< l_j'\}} \rcoeff{\vec{r}}
  E_{\divided{\vec{r}}} \tensor F_{\divided{\vec{r}}} \in u_\s^{\tensor 2}
\end{align*}
where $\rcoeff{\vec{r}}$ is as in \autoref{eq:largermatrix} and
$g=K_{2\rho}=\enrich(K_{2\rho}) \in u_\q$.

\begin{remark}
  $u_\q$ is essentially the same algebra as the one defined by
  \cite{Lusztig93} and other authors.  However, this $R$ matrix is not
  the same as the one defined by Rosso (\cite{Rosso90}) and
  \cite{Lusztig93}, which is in $u_\q$ not $u_\s$.  For most purposes
  these $R$ matrices function interchangeably, but here being the
  image of the $\Rlarge$-matrix for the large quantum group is a crucial
  property, while being in $u_\q$ will not prove important.
\end{remark}

\begin{proposition} \label{pr:phi}
  if $T$ is a bottom tangle (\autoref{p:tangle}) with all closed components
  labeled by elements of $U_\s^{*,\invariant}$, then
  \begin{gather}
    \enrich\parens{\universal_{\Ufull}(T) } =
    \universal_{u_\s}(T) \label{eq:large2small}
  \end{gather}
where the invariant functionals are  interpreted as the
restriction to $u_\s$ on the right side and the pullback by $\enrich$ on
the left.  
\end{proposition}

\begin{proof}
   This is immediate from the fact that $\enrich$ is a homomorphism which
   preserves the ribbon structure.  
\end{proof}

\subsection{Representations, Invariant  Functionals and the Center}
\label{ss:reps}

\subsubsection{Representations}\label{p:reps}
From \cite{Rosso90,CP94} every  finite-dimensional representation of
$\Uq$ extended to an algebra over $\CC(q)$ is a finite direct sum of finite-dimensional
simultaneous eigenspaces for $\{K_i\}$ with $E_i$ and $F_i$ acting
nilpotently.  More precisely, for each  weight 
$\lambda$ in the
\emph{Weyl chamber} $\lattice^+$, i.e. $\lambda \in \lattice$ satisfying
$\bracket{\lambda+\rho,\alpha_i} > 0$ for all $i$, there is  the \emph{Weyl module}
$V_\lambda$, a representation of
 $\Uq$ over 
$\A_q$.  Every representation over $\CC(q)$ is a sum of simple representations which are
each a tensor product of a
one-dimensional representation with the extension to $\CC(q)$ of a Weyl module.  One can write
$V_\lambda = \bigoplus_{\mu \in \lattice} V_{\lambda, \mu} $ where
$K_iv= q_i^{\bracket{\alpha_i,\mu}}v$ for $v \in V_{\lambda,\mu}$ and
  $E^{\divided{r}}_\beta\colon V_{\lambda, \mu} \to V_{\lambda,
    \mu+r\beta}$, $F^{\divided{r}}_\beta\colon V_{\lambda, \mu} \to V_{\lambda,
    \mu-r\beta}$, and the only $\mu$ with nonzero contributions are in
  Weyl orbits of $\mu \in \lattice^+$ with $\lambda-\mu$ a sum of
  positive roots.  In fact the Weyl modules extend naturally to
  $\Ufull$ and $\Ufullcomplete$, while the nontrivial one-dimensional
  representations straightforwardly do not.  More generally 

  \begin{lemma}
    If $\lambda_1, \ldots, \lambda_m\in \fullchamber$ then the action of
    $\Uqm$ on $V_{\lambda_1} \tensor V_{\lambda_2} \tensor
    \cdots \tensor V_{\lambda_m}$ extends to a continuous action of
    $\Ufullcompletem$ on the corresponding module over
    $\Acomplete_s$.  In particular
    $\universal_{\Ufull}(T)$ for $T$ a bottom tangle with $m$ open components  acts on
    any $m$-fold tensor product of direct sums of Weyl modules.
  \end{lemma}

  \begin{proof}
    Define the action of $U^{\tensor m}_{\full,0}$ on  $V_{\lambda_1} \tensor V_{\lambda_2} \tensor
    \cdots \tensor V_{\lambda_m}$ by declaring  $[\gamma_1] \tensor \cdots \tensor [\gamma_m]$ to act as $\delta_{\gamma_1,\mu_1} \cdots \delta_{\gamma_m,\mu_m}$ on $V_{\lambda_1,\mu_1} \tensor 
    \cdots \tensor V_{\lambda_m,\mu_m}$.  This  defines uniquely an action of $\Ufullm$ on $V_{\lambda_1} \tensor V_{\lambda_2} \tensor
    \cdots \tensor V_{\lambda_m}$ which extends the action of $\Uq^{\tensor m}$.  Since all but finitely many $[\gamma_1] \tensor \cdots \tensor [\gamma_m]$ act as $0$, this action is continuous.  Finally we extend the action to $\Ufullcompletem$ over $\Acomplete_s$.
  \end{proof}

\subsubsection{Representations of the small quantum group}\label{p:smallreps}
 
  Passing to a root of unity, one can extend the module $V_\lambda$ 
  over $\A_s$ or $\A_q$  to a module over $\QQ(\s)$ or $\QQ(\q)$ and it
  becomes a representation of the image of $U_\s$, $u_\s$ and $u_\q$,
  all with the same weight decomposition and all also called
  $V_\lambda$. Tensor products of such representations are
  representations of the appropriate tensor power of algebra (without
  need of completion) and therefore of the original algebra through
  $\Delta$.    If $\lambda$ is in the \emph{Weyl
    alcove} $\lattice^l$, i.e. it is in the Weyl chamber $\fullchamber$ and satisfies
  $\bracket{\lambda+\rho, \phi} \leq l'$ where $\phi$ is the long
  positive root if $\max(d_i) | l'$ and the short positive root
  otherwise, then it is a simple representation and is the orbit of
  the highest weight vector in $u_\s$ and $u_\q$ \cite{Sawin06a}.

\subsubsection{Quantum traces} \label{p:qtr}
  In particular for $\lambda \in \fullchamber$, define $\qtr_\lambda(a)= \tr_\lambda(K_{2\rho}a)$
  where $\tr_\lambda$ is the trace in the representation $V_\lambda$ (\autoref{p:reps}).
  This is an invariant functional, and thus the universal invariant
  (\autoref{p:universal}) for
  $\Ufullcomplete$ is defined for bottom tangles (\autoref{p:tangle}) where some closed
  components are labeled by $\qtr_\lambda$.

\subsubsection{The Drinfel'd map} \label{p:drinfeld}
  Recall $\clasp \in U^{\tensor 2, \invariant}$ is the universal
  invariant (\autoref{p:THA-definition})
   of the clasp tangle in
  \autoref{fg:some-tangles}.  In particular note that for $\Ufull$
  \begin{align}
    \clasp&=\sum_{\vec{r},\vec{s}}\rcoeff{\vec{r}}\rcoeff{\vec{s}} q^{-2\bracket{\rho,\theta_{\vec{s}}}}
    \sum_{\gamma} 
    F_{\divided{\vec{r}}} K_{2\gamma-\theta_{\vec{s}}}E_{\divided{\vec{s}}} \tensor
    F_{\divided{\vec{s}}} \bgam K_{\theta_{\vec{r}}}E_{\divided{\vec{r}}} \label{eq:clasp-left}\\
    &=\sum_{\vec{r},\vec{s}}\rcoeff{\vec{r}}\rcoeff{\vec{s}} q^{-2\bracket{\rho,\theta_{\vec{s}}}}
    \sum_{\lambda} 
    F_{\divided{\vec{r}}} \blam K_{\theta_{\vec{s}}}E_{\divided{\vec{s}}} \tensor
    F_{\divided{\vec{s}}} K_{2\lambda} K_{-\theta_{\vec{s}}}E_{\divided{\vec{r}}}. \label{eq:clasp-right}
  \end{align}
  and if $\clasp_-$ is the universal invariant of the clasp tangle with all crossings reversed then
  \begin{align}
    \clasp_-&=\sum_{\vec{r},\vec{s}}\rcoeff{\vec{r}}\rcoeff{\vec{s}} q^{-\bracket{\theta_{\vec{s}},\theta_{\vec{s}}}}
    \sum_{\gamma} 
    S\parens{E_{\divided{\vec{r}}}} \bgam K_{-\theta_{\vec{s}}}S\parens{F_{\divided{\vec{s}}}} \tensor
    E_{\divided{\vec{s}}}  K_{-\theta_{\vec{s}}} K_{-2\gamma} F_{\divided{\vec{r}}} \label{eq:clminus-left}\\
    &=\sum_{\vec{r},\vec{s}}\rcoeff{\vec{r}}\rcoeff{\vec{s}}  q^{-\bracket{\theta_{\vec{r}},\theta_{\vec{r}}}}
    \sum_{\lambda} 
    S\parens{E_{\divided{\vec{r}}}}  K_{-2\lambda} K_{\theta_{\vec{r}}} F_{-\divided{\vec{s}}} \tensor
    E_{\divided{\vec{s}}} K_{\theta_{\vec{s}}} \blam F_{\divided{\vec{r}}}. \label{eq:clminus-right}
  \end{align}

$C$
  defines an algebra homomorphism called the \emph{Drinfeld map} from
  invariant functionals to the center sending each $\psi \in
  \overline{U}^{*,\invariant}$ to
  \[\DD\sbrace{\psi}=(1 \tensor \psi)(\clasp) \in \overline{U}^{\invariant}.\]%

Notice in the special case of $\Ufull$ the image of $\DD$ is invariant
elements of the even part of $\Uweightcomplete$.

\subsubsection{The even center}\label{p:zlambda}
  If $\lambda \in \fullchamber$ define $z_\lambda \in
  \Ufull^{\invariant}$ or in $u_\s^{\invariant}$ by
  \begin{gather}
    \label{eq:zlambda}
    z_\lambda= \DD\sbrace{\qtr_\lambda}.
  \end{gather}%
  In fact by \autoref{eq:clasp-left} $z_\lambda \in U_\weight^{\even,
   \invariant}$ and if $\theta \in \lattice_r\cap \lattice^+$ then
  $z_\theta \in \Uq^{\even,\invariant}$. Define $Z=\Uq^{\invariant}$
  and $Z^\even=\Uq^{\invariant,\even}$.

  In particular  in $u_\s$, $z_\lambda = (1 \tensor
  \qtr_\lambda)\universal_{u_\s}(\clasp)=( 1 \tensor \qtr_\lambda)
  \enrich\parens{\universal_{\Ufull}(\clasp)} = \enrich(z_\lambda).$

\subsubsection{The Harish-Chandra map} \label{p:HC}
  \autoref{eq:q-tensor-basis}  implies any element of $\Uqcomplete$ can be written as a
  sum $F_{\divided{\vec{r}}} H_{\vec{r},\vec{s}} E_{\divided{\vec{s}}}$ with
  $H_{\vec{r},\vec{s}} \in \Uqz$ and for each $k,i$ all
   but finitely many $\vec{r},\vec{s}$ have $H_{\vec{r},\vec{s}}$
   in $\Ik$.  Thus one gets the uniformly continuous Harish-Chandra
  map $\HC \colon \Uqcomplete \to \Uqzcomplete$ which sends each element of $\Uqcomplete$
  to  $H_{\vec{0},\vec{0}}$.  It is easy to
  check that this map is an algebra homomorphism on the center.  De
  Concini and Kac \cite{DeconciniKac90} prove that this map acting on
  $\Uq \tensor_{\A_q} \CC(q)$ (which \cite{CP94} calls the rational
  form) is a bijection from
  the center  to the translated Weyl-group-invariant portion of the
  Cartan subalgebra (\autoref{p:lie-alg}).

\begin{lemma}\label{lm:zeven}
The left degree $0$ part $Z^{\even}$ of the center $Z$ of $\Uq$ is
spanned by $\{z_\theta\}_{\theta \in \lattice_r}$. 
The left degree $0$ part $\widehat{Z}^{\even}$ of the center $\widehat{Z}$ of $\Uqcomplete$ is
spanned topologically by the $\{z_\theta\}_{\theta \in \lattice_r}$.
\end{lemma}

\begin{proof} Since
  the Harish-Chandra map preserves left degree (\autoref{p:divided}) it is also a bijection
  when restricted to a map between the  even parts of the
  center and of the translated  Weyl-invariant Cartan algebra of the rational form.   Check
  following Jantzen \cite{Jantzen96} that $\HC(z_\lambda)= \sum_{\mu}
  \dim_\lambda(\mu) K_{2\mu}$, where $\dim_\lambda(\mu)$ is the
  dimension of the weight $\mu$ subspace of $V_\lambda$.  Clearly for
  each $\theta \in \lattice_r \cap \lattice^+$ one can form a linear combination
  $z'_\theta$ of $z_\mu$ for $\mu \leq \theta$ such that
  $\HC(z_{\theta}')= \sum_{\sigma \in W}
  K_{\sigma(\theta)}$, which implies that the $z_{\theta}'$ and hence
  $z_\theta$  span the even
  part of the center of the rational form.  Since each $z_\theta \in
  Z^{\even}$, it follows that $Z^{\even}$ is spanned by these
  elements in $\Uq$.

Given $\lambda \in \lattice$ consider the right ideal in $\Uqcomplete$
generated by $\{E_i\}$ and by $k-\lambda(k)$ for each $k \in
\Uqzcomplete$. $\Uqcomplete$ acts on the left on the quotient by this ideal.  Notice
that $z \in \overline{Z}$ acts by multiplication by the scalar
$\lambda(\HC(z))$, since it acts on the image of $1$ 
as that.  On the other hand if $\bracket{\alpha_i, \lambda}=n>0$ then
the image of $F_i^{n+1}$ in this module is easily recognized to be a
highest weight module, and therefore $Z$ acts on it by multiplication
by $\sigma_i \lambda (\HC(z))$, where $\sigma_i \lambda = \lambda - (n+1)
\alpha_i$ is the image of $\lambda$ under the appropriate generator
of the translated Weyl group.  This is to say that $\lambda$ and
$\gamma$ agree on $\HC(z)$ if they are linked by an element of the
translated Weyl group.  Since $K \in \Uqzcomplete=0$ if all
$\lambda(K)=0$ for $\lambda \in \lattice$ (because this is true for
$\Uqz$) this implies that $\HC(\widehat{Z})$ is invariant under the translated 
Weyl group action. 

Therefore modulo any $\Ik$, $\HC(z)$ for $z \in \widehat{Z}^\even$
is equivalent to an invariant 
element of $\Uqz$, which is $\HC(z')$ for $z'$ a linear combination of
$\{z_\theta\}$, and therefore $z$ is equivalent to $z'$ modulo
$\Ik$.  So $\widehat{Z}^\even$ is in fact the completion of $Z^\even$.
\end{proof}

\begin{remark}
The argument above readily generalizes to show that the even part of
the center of $\Uweightcomplete$ is spanned by $z_\lambda$ for
$\lambda \in \lattice$.  This fact will not be important in this paper.
 \end{remark}

\section{Chern-Simons and Hennings Invariants}

For this section fix a positive integer $l$  and hence roots of unity
$\s$ and $\q$ (\autoref{p:small}).  Most of this section will work with $U_\s$, $u_\s$ and
$u_\q$, but in some cases conclusions will be drawn about $\Ufull$.
  \subsection{The CS and Hennings Invariant}
\label{ss:cs-hen}

\subsubsection{The Chern-Simons invariant  functional} \label{p:cs}
  Let $\lattice_H$ be a sublattice of $\lattice$ containing $\lattice_r$
  ($\lattice/\lattice_r$ is cyclic except in the case $D_4$.  In that
  case do not consider the case
  $\lattice_H=\lattice_r$). Let $\lattice^l_H$ be the intersection of the Weyl
  alcove with $\lattice_H$, i.e. the set of all
  $\lambda \in \lattice^+_H$  such that $\bracket{\lambda+ \rho,\phi}
  \leq l'$ where $\phi$ is the long positive root if $\max(d_i)|l'$ and
  the short positive root otherwise.   Define the
  invariant functional 
  \begin{gather}
    \label{eq:cs-label}
    \cs_H = \sum_{\lambda \in \lattice^l_H} \qtr_\lambda(1) \qtr_\lambda.
  \end{gather}%
  This expression can be interpreted as an invariant functional on
  $\Ufull$, $U_\s$ or $u_\s$, and the map $\enrich$ intertwines these
  functionals.  The quantum Racah formula \cite{Sawin06a}[Cor. 8]  shows that if $\phi$ is a linear
  combination of quantum traces $\qtr_\lambda$ for $\lambda \in
  \lattice_H^l$ then in $U_\s$  the 
  handle-slide condition holds
  \begin{gather}
    \label{eq:cs-handleslide}
   \sbrace{ \phi  \cs_H }(x)= \phi(1) \cs_H(x)
  \end{gather}%
  \begin{figure}[ht]
    \centering
    \includegraphics[width=5.5in]{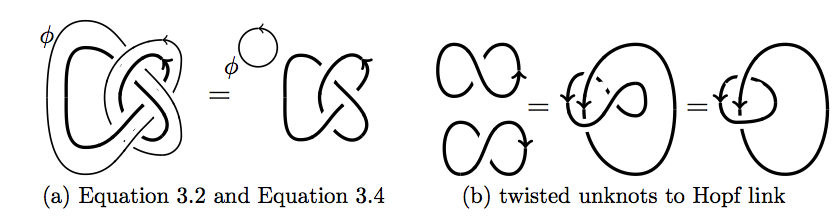}
    \caption{Handle slides - thick components are labeled by
    $\cs_H$ or $\hen$} \label{fg:handle-slide}
    \end{figure}
which on the level of invariants says that any component labeled by a
linear combination of quantum traces in the sublattice can slide through a component
labeled by $\cs_H$ as in \autoref{fg:handle-slide}.
  Consider any pair $l,H$ such that the invariant $\cs_H\parens{\DD[\cs_H]}
  \cs_H(\clasp)$ of the Hopf link with both components labeled by $\cs_H$
  is nonzero.  \autoref{eq:cs-handleslide} then implies that
  $\cs_H(\twist^{\pm  1})$ are nonzero.
  \cite{Sawin02a,Sawin06b,Sawin06a} 
  identify a large array of pairs for which this holds, including all
  the cases corresponding to levels and Lie groups identified by
  Dijkgraaf and Witten \cite{DW90} as admitting a Chern-Simons field
  theory.   In 
  particular this holds when $l$ is divisible by $2 \max(d_i)$ and
  $H=\{0\}$.  In any such case 
  \begin{gather}
    \label{eq:cs-tmi}
    \cs_H^{\tensor m}\parens{\universal_{u_{\s}}(T)} \cs_H(\twist)^{-\sigma_+} \cs_H(\twist^{-1})^{-\sigma_{-}}
  \end{gather}%
  where $\sigma_\pm$ are the number of positive and negative
  eigenvalues of the linking matrix, is an invariant of the
  three-manifold given by surgery on the closure of an open bottom
  tangle $T$. 

  Observe that $\qtr_\lambda$ can be defined as an invariant  functional
  on the small and large quantum group and that these are preserved by
  $\enrich$.  Thus $\cs_H$ can be interpreted as an invariant functional on
  either quantum group.  In particular  $\cs_H^{\tensor
    m}\parens{\universal_{\Ufull}(T)}$ makes sense in
  $\Acomplete_s$, and when extended to $\CC(\s)$ agrees with
  $\cs_H^{\tensor m}\parens{\universal_{u_{\s}}(T)} $ and thus
  \autoref{eq:cs-tmi} with this replacement defines the same
  three-manifold invariant.

\subsubsection{The Hennings invariant  functional} \label{p:hennings}
  Let $\vec{r}_{\text{max}}= (l'_{\beta_1}-1, \ldots,
      l'_{\beta_N}-1)$ and $\theta_{\text{max}}= \sum_j (l'_{\beta_j}-1)
      \beta_j$.  Notice that in $u_\s$ and $u_\q$ we have 
      $E_iE_{\divided{\vec{r}_{\text{max}}}}=E_{\divided{\vec{r}_{\text{max}}}}E_i=0$ for
      all $i$.  The small quantum group $u_{\s}$ has a \emph{left
        integral} \cite{Hennings96}  which is an
  invariant functional 
  \[\hen\parens{F_{\divided{\vec{r}}}K_\theta E_{\divided{\vec{s}}}}=
  \begin{cases} 1 & \text{ if } \vec{r}=\vec{s} =
     \vec{r}_{\text{max}} \text{ and } \theta =  \theta_{\text{max}}\\
      0 & \text{ else.}
    \end{cases}\]%
  with again the handleslide property the property that
  \begin{gather}\label{eq:hen-handleslide}
\sbrace{\phi \hen}(x)= \phi(1) \hen(x)
\end{gather}%
  for $\phi \in u_{\s}^*$.  An easy calculation gives that the
  invariant of the Hopf link
  $\hen\parens{\DD\sbrace{\hen}}$, and therefore $\hen(\twist^{\pm 1})$ are
  all nonzero, and thus once again 
  \begin{gather}
    \label{eq:hen-tmi}
    \hen^{\tensor m}\parens{\universal_{u_{\s}}(T)} \hen(\twist)^{-\sigma_+} \hen(\twist^{-1})^{-\sigma_{-}}
  \end{gather}
 is an invariant of the
  three-manifold given by surgery on the closure of  $T$. 
%

\subsection{The Universal Invariant of ZLMTs}

A  \emph{zero linking matrix tangle} or ZLMT is a  bottom tangle in which the
self-linking number of each open component and the linking number between
each pair of components are all zero, and in which each closed
component is labeled by an invariant functional (if there are no
closed components call it an \emph{open ZLMT}).  These are also called algebraically split tangles.  On the one hand, this section
argues that this condition restricts the possible values of the
universal invariant (\autoref{p:universal}) to even elements of $\Uqcomplete$.  On the other hand, the universal invariant of such tangles can
be used to compute the CS and Hennings invariants
(\autoref{p:cs},\autoref{p:hennings}) of homology
three-spheres.

Writing out a state of the universal $\Ufullcomplete$ invariant of a projection of
a ZLMT, the algebra of the tensor $\Rlarge_0$ of 
\autoref{eq:hrmatrix} and its relation with the other elements of the
state are quite straightforward and one sees readily that the result
is a tensor product of elements of $\Uqcompletem$.  More
care is required to see that the universal invariant is in fact even
in the sense of the previous section.

\begin{figure}[ht]
  \centering
  \includegraphics[width=4.5in]{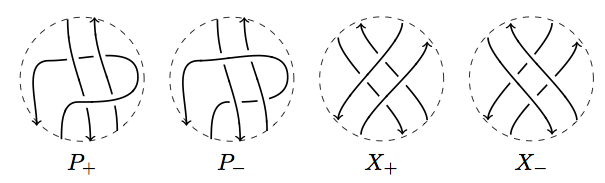}
\caption{fragments generating ZLMTs}
\label{fg:generators}
\end{figure}

\begin{lemma}
  \label{lm:ZLMT}
  If $T$ is an open zero linking matrix tangle with $m$ components, then $T$ has a presentation that is without crossings except for some number of fragments that look like $P_\pm$, $X_\pm$ as in \autoref{fg:generators}, where in each case the doubled strand always represents pieces of the same component.  Further, the total rotation of each component from beginning to end is $(4n+1) \pi$.
\end{lemma}

\begin{proof}
  According to \cite{Habiro06}[Cor. 9.13], any open  ZLMT can be written as a product
$WB^{\tensor k} $, where $B^{\tensor k}$ is a horizontal product of the Borromean
tangle $B$ in \autoref{fg:borromean}, and $W$ is a horizontal and
vertical product of the tangles $1_b,\mu_b$, $\eta_b$, $\gamma_+$,
$\gamma_-,\psi_{b,b}, \psi_{b,b}^{-1}$ in \autoref{fg:base-tangles}.  It is apparent in \autoref{fg:borromean} that the Borromean tangle can be decomposed into this form, and clearly those in \autoref{fg:base-tangles} are already in this form.  For the last sentence, notice this is equivalent to the statement that if we complete the component to a closed component by adding a rotation $\pi$ cup the winding number is odd.  But the winding number and the self-linking number (writhe) are of opposite parity \cite{Sawin02b}, and the writhe of each  component of a ZLMT is $0$.
\end{proof}

\begin{proposition} \label{pr:ZLMT}
If $T$ is an open zero linking matrix tangle with $m$ components, then
$\universal_{\Ufull}(T) \in
\Uqcomplete^{\even,\invariant,\tensor m}$.
\end{proposition}

\begin{proof}
This follows directly from \autoref{lm:weak}.  The computation of the universal invariant of $T$ as in \autoref{p:universal} and \autoref{p:properties} can be organized as follows:  in a presentation as in \autoref{lm:ZLMT}, compute the invariant (as sums over states of algebra elements labeling strands) of each piece that looks like \autoref{fg:generators} as well as each cup and cap as in \autoref{fg:universal}.  The overall invariant will be a sum over states of a  product for each component of all the given labels in some order. Thus by \autoref{lm:weak} it suffices to show that for each state the tensor product of the labels is weakly even, which is to say consists of some number of even elements, an even number of odd elements, and some number of elements in $A\sbrace{\Ufullcomplete}$ with $A=(1\tensor S) \Delta$.  Since the winding number of each component is even, the product will have an even number of factors of $K_{2\rho}$ which is in $\Uq^\even$ or $\Uq^{\mathrm{odd}}$.  Each $\psi^{\pm 1}$ will contribute $A(a_k) \tensor A(b_k)$ or $A(\tilde{a}_k) \tensor A(\tilde{b}_k)$  (by Properties 1 and 2 of \autoref{p:properties}) where $R=\sum_k a_k \tensor b_k$, $R^{-1} =\sum_k \tilde{a}_k \tensor \tilde{b}_k$.  Similarly, the invariant of $\gamma_+$ is $1 \tensor A$ applied to the invariant of the clasp (more precisely, the clasp without one of its caps, but the invariant is the same).  Notice in \autoref{eq:clasp-left} the invariant of the clasp is a sum of elements  $\Uq^\even \tensor \Ufull$, and therefore for each term in the sum the invariant assigns an even factor to one component and a pair of factors $A \sbrace{\Ufullcomplete}$ to the other. This completes the proof.

\begin{figure}[ht]
  \centering
  \includegraphics[width=4in]{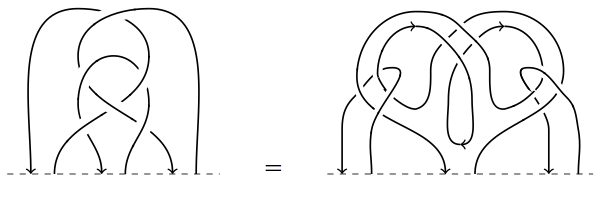}
\caption{The Borromean Tangle}
\label{fg:borromean}
\end{figure}

\begin{figure}[ht]
  \centering
  \includegraphics[width=5.5in]{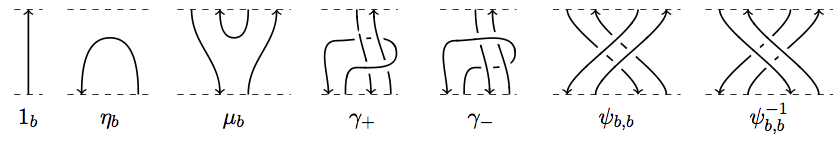}
\caption{Tangles which with Borromean generate ZLMTs}
\label{fg:base-tangles}
\end{figure}

\end{proof}

\begin{theorem} \label{th:zlmt-zlambda}
  $T$ is a mixed ZLMT with closed components labeled by invariant
  functionals on $\Uq$ or $\Ufullcomplete$ then $\universal_{\Ufull}(T) \in
  \Uqcomplete^{\even,\invariant,\tensor m}$.  If the components are labeled
  by invariant functionals on $u_\q$ or $u_\s$ (\autoref{p:small}) then $\universal_{u_{\s}}(T) \in
  u_\q^{\even,\invariant,\tensor m}$.  If $T$ has a single open component
  then in either case it is a linear combination (in the first case
  possibly infinite) of $\{z_\theta\}_{\theta \in \lattice_r}$.
\end{theorem}

\begin{proof}
  The first sentence follows from \autoref{pr:ZLMT} because applying
  an invariant functional to
  one factor of $\Uqcomplete^{\even, \invariant, \tensor m}$ maps into
  $\Uqcomplete^{\even, \invariant, \tensor m-1}$.  The second sentence
  follows from this, \autoref{pr:phi} and the fact that $\enrich$
  maps  $\Uqcomplete^{\even, \invariant, \tensor m}$ to $u_\q^{\even,\invariant,
    \tensor m}$. The third sentence follows from  \autoref{lm:zeven}. 
\end{proof}


\subsection{Equality of Invariants on Homology Spheres}

In this section assume that $l$ and $H$ are such that $\cs_H(\twist^{\pm
  1})$ are nonzero  (\autoref{p:cs}).

\begin{lemma} \label{lm:hcs-zlambda}
Let $\lambda \in \lattice$. Then
\begin{gather} \label{eq:hcs-zlambda}
\cs_H(z_{\lambda}\twist^{\pm 1}) \hen(\twist^{\pm 1})= \hen(z_{\lambda}\twist^{\pm 1}) \cs_H(\twist^{\pm 1}).
\end{gather}
\end{lemma}

\begin{proof}
Recall that since $\twist^{-1}$ is central it acts on $V_\lambda$
(\autoref{p:reps}) as a multiple of
the identity, so that $\qtr_\lambda(\twist^{-1} \,\cdot\,)$ is a multiple of
$\qtr_\lambda$.  Thus by
 \autoref{eq:cs-handleslide} and the fact that $\Delta \twist = (\twist
\tensor \twist) \clasp$ 
\begin{align*}
\qtr_\lambda (\twist^{-1}) \cs_H (\twist) &=
\sbrace{\qtr_\lambda(\twist^{-1}\,\cdot \, ) 
  \cs_H}(\twist) \\
&=\sbrace{\qtr_\lambda
\tensor \cs_H(\twist\,\cdot \,)}(\clasp)
= \cs_H( z_\lambda \twist).\\
\end{align*}
The same is true replacing  $\hen$ for $\cs_H$, and
\autoref{eq:hen-handleslide} for \autoref{eq:cs-handleslide}.  Multiplying the resulting equations by
$\hen(\twist)$ and $\cs_H(\twist)$ gives
\[\cs_H(z_\lambda \twist)\hen(\twist)=\cs_H(\twist)\qtr_\lambda(\twist^{-1})\hen(\twist)=\cs_H(\twist)\hen(z_\lambda \twist).\]
The same argument works replacing $r$ and $r^{-1}$, using the inverse
of the clasp element, and replacing $\lambda$ with the highest weight
$\lambda^*$ of the dual representation.
\end{proof}

\begin{lemma} \label{lm:hcs-onetangle}
  Let $T$ be a ZLMT with every component but one closed off and
  labeled by a quantum functional on $u_\s$.  Let $\epsilon=\pm 1$. Then
  \begin{gather}
    \label{eq:hcs-onetangle}
    \frac{\cs_H\sbrace{\twist^{\epsilon}
          \universal_{u_\s}(T)}}{\cs_H\sbrace{\twist^\epsilon}} = 
    \frac{\hen\sbrace{\twist^{\epsilon}
          \universal_{u_\s}(T)}}{\hen\sbrace{\twist^\epsilon}}
  \end{gather}
\end{lemma}

\begin{proof}

  By \autoref{th:zlmt-zlambda} $\universal_{u_\s}(T)$ is a linear
  combination of $z_\theta$ for $\theta \in \lattice_r$.  Thus
  \autoref{eq:hcs-zlambda} holds with $z_\lambda$ replaced by
  $\universal_{u_\q}(T)$, and dividing by the nonzero quantities gives \autoref{eq:hcs-onetangle}.
\end{proof}

\begin{proposition}\label{pr:hcs-zlmt}
  Let $T$ be an open ZLMT with $m$ components and let $\epsilon_1, \ldots , \epsilon_m$ be
  nonzero integers with $|\epsilon_i|=1$ for all $1 \leq i \leq m$.  Then 

\begin{gather}\label{eq:hcs-zlmt}
\frac{\cs_H^{\tensor m}[(\twist^{\epsilon_1} \otimes \twist^{\epsilon_2} \otimes \cdots \otimes \twist^{\epsilon_m}) \universal_{u_q}(T)]}{\cs_H(\twist^{\epsilon_1})\cs_H(\twist^{\epsilon_2})\cdots\cs_H(\twist^{\epsilon_m})} =  \frac{\hen^{\tensor m}[(\twist^{\epsilon_1} \otimes \twist^{\epsilon_2} \otimes \cdots \otimes \twist^{\epsilon_m}) \universal_{u_q}(T)]}{\hen(\twist^{\epsilon_1})\hen(\twist^{\epsilon_2})\cdots\hen(\twist^{\epsilon_m})}
\end{gather}
\end{proposition}

\begin{proof}
For each $0 \leq k \leq m$ let $A_k$ be  the left-hand side of
\autoref{eq:hcs-zlmt} with the $\cs_H$ replaced by $\hen$ in each factor
in the top and bottom from the first to through the $k$th.  Thus the
left hand side corresponds to $A_0$ and the right hand side to $A_m$.
So it suffices to prove that $A_{k-1}=A_{k}$ for $1 \leq k \leq m$.
If $T_k$ is  $T$ with component $j$ for $j <k$ closed and
labeled by  $\cs_H(\twist^{\epsilon_j} \, \cdot \,)$ and component  $j$ component for
$j>k$ closed and labeled by $\hen(\twist^{\epsilon_j} \, \cdot \,)$, then 
$A_k = A \hen
\sbrace{r^{\epsilon_k}\universal(T_k)}/\hen\sbrace{\twist^{\epsilon_k}}$ and
$A_{k-1}= A \cs_H
\sbrace{r^{\epsilon_k}\universal(T_k)}/\cs_H\sbrace{\twist^{\epsilon_k}}$ where
\[A^{-1} = \hen\sbrace{\twist^{\epsilon_1}} \cdots
\hen\sbrace{\twist^{\epsilon_{k-1}}} \cs_H\sbrace{\twist^{\epsilon_{k+1}}} \cdots
\cs_H\sbrace{\twist^{\epsilon_m}}.\]
Thus $A_k=A_{k-1}$ is exactly~\autoref{eq:hcs-onetangle}.
\end{proof}

\begin{theorem}\label{th:hcs}
  When $l$ and $H$ are such that the Chern-Simons invariant is defined
  (i.e. if $\cs_H(\twist^{\pm 1}) \neq 0$),
  The Hennings invariant of an integral homology three-sphere $M$ is the
  Chern-Simons invariant of $M$.
\end{theorem}

\begin{proof}
$M$ can be obtained by (integral) surgery in $S^3$ on the a framed
link whose linking matrix has all zeros except for $\pm 1$ along the
diagonal. The invariant of such a link is  a product of invariant functionals applied to
$\parens{\twist^{\epsilon_1} \tensor \cdots \tensor
  \twist^{\epsilon_m}}\universal_U(T)$ where $T$ is an open ZLMT.  Therefore the Chern-Simons invariant
of $M$ is given by the left hand side of \autoref{eq:hcs-zlmt} and the
Hennings invariant is given by the right-hand side.
\end{proof}

\bibliographystyle{alpha}

\def\cprime{$'$} \def\cprime{$'$}

\end{document}